\documentclass{amsart}

\usepackage{ifthen,amsthm,amssymb,latexsym,amsmath}

\usepackage{graphicx}

\input xy

\xyoption{all}

\usepackage[all]{xy}

\newcommand{\cb}{\overline{C}}

\newcommand{\cF}{\mathcal{F}}

\newcommand{\mmp}{{\mathfrak m}_P}

\newcommand{\mmpb}{{\mathfrak m}_{\overline{P}}}

\newcommand{\oo}{{\mathcal O}}

\newcommand{\op}{{\mathcal O}_P}

\newcommand{\obp}{\overline{\mathcal O}_P}

\newcommand{\pb}{\overline{P}}

\newcommand{\ra}{\rightarrow}

\newcommand{\cE}{\mathcal{E}}

\newcommand{\cM}{\mathcal{M}}

\DeclareMathOperator{\Hom}{{Hom}}

\DeclareMathOperator{\Ext}{{Ext}}

\newtheorem{coro}{Corollary}[section]

\newtheorem{lemm}[coro]{Lemma}

\newtheorem{rema}[coro]{Remark}

\newtheorem{theo}[coro]{Theorem}

\numberwithin{equation}{section}

\newcommand{\marginnote}[1]{\ifthenelse{\isodd{\thepage}}{\normalmarginpar}

{\reversemarginpar}\marginpar{\fbox{\parbox{24mm}{\sloppy\footnotesize #1}}}}

\begin{document}
\title{Torsion Free Sheaves on Cuspidal Curves}

\author{Dan Avritzer}

\author{Flaviana Andrea Ribeiro}

\author{Renato Vidal Martins}

\address{ICEx - UFMG \\
Departamento de Matem\'atica \\
Av. Ant\^onio Carlos 6627 \\
30123-970 Belo Horizonte MG, Brazil}

\email{dan@mat.ufmg.br}

\address{F.A. Ribeiro\\Departamento de Matem\'atica,
UFJF, Juiz de Fora, MG 36036-330, Brazil}
\email{flaviana.ribeiro@ufjf.edu.br}

\address{ICEx - UFMG \\
Departamento de Matem\'atica \\
Av. Ant\^onio Carlos 6627 \\
30123-970 Belo Horizonte MG, Brazil}

\email{renato@mat.ufmg.br}

\begin{abstract}

We study torsion free sheaves on integral projective curves with at most ordinary cusps as singularities. Adjusting Seshadri's structure from the nodal case to this one, we describe these sheaves by means of a triple defined in the normalization of the curve.

\end{abstract}

\maketitle

\section{Introduction}

In \cite{S}, C. S. Seshadri characterized a torsion free sheaf on a
nodal curve in terms of a triple data in its normalization. More
recently, in \cite{ALR}, H. Lange along with the first and second
named authors gave to Seshadri's result another proof within a
categorical framework, besides extending the discussion to
stability. The aim of this article is to answer whether it is
possible to apply the methods in \cite{ALR} to get a structure like
in \cite{S} to describe torsion free sheaves on cuspidal curves. So
our technique, as in \cite{ALR}, differs from the other papers on
the subject as \cite{b,bd,bdg}.

More precisely, let $C$ be a \emph{curve}, i.e., an integral and
complete one dimensional scheme over an algebraically closed field.
Let also $\overline{C}$ be its normalization. Assume, first, that
$C$ has just one singular point $P$ which is an ordinary node and
let $\pb_1$ and $\pb_2$ be the points in $\cb$ over $P$. In
\cite[Ch. 8]{S}, Seshadri totally described a torsion free sheaf
$\cF$ on such a $C$ in terms of three data: (i) a bundle $E$ on
$\cb$; (ii) a pair $(V_1,V_2)$ of vector subpaces of the fibers
$E_{(\pb_1)}$ and $E_{(\pb_2)}$ respectively; (iii) a linear
isomorphism $\sigma: V_1\to V_2$. This characterization naturally
extends to curves allowing nodes at most as singularities, no matter
how many. We think it is important to say that a key point of his
proof is certainly \cite[Prp. 2, p. 164]{S} which gives a local
description at the singular point, that is, $\cF_P\simeq \op^{\oplus
a}\oplus\mathfrak{m}_P^{\oplus r-a}$ where $r$ is the rank of $\cF$
and $a$ is an integer depending on $\cF$ which, actually, determines
the dimension of $V_1$ and $V_2$ above.

In \cite{ALR}, a categorical structure was given to this set of
triples, and it was proved that the assignment $\cF\mapsto
(E,(V_1,V_2),\sigma)$ is functorial. Here, we want to establish a
similar result when $C$ is cuspidal.

Before stating the results we have, we briefly introduce our set-up:
a point $P\in C$ is said to be a (first order) \emph{cusp} if it is
unibranch, i.e., there is a unique point $\pb\in\cb$ with $\pb$ over
$P$ and $\dim(\obp/\op)=1$. For a torsion free sheaf $\cF$ of rank
$r$ on $C$, one defines its degree as
$\deg(\cF)=\chi(\cF)-r\,\chi(\oo_C)$ where $\chi(\cF)$ denotes the
Euler characteristic of $\cF$. With this in mind, we have the
following.
\

\noindent {\bf Theorem.} \emph{Let $C$ be a curve with singular points $P_1,\ldots,P_n$, which are all cusps. Then the following hold:
\begin{itemize}
\item[(1)] If $\cF$ is a torsion free sheaf of rank $r$ on $C$, then
$$
\cF_{P_i}\simeq \oo_{P_i}^{\oplus a_i}\oplus\mathfrak{m}_{P_i}^{\oplus r-a_i}
$$
for $1\leq i\leq n$. Each integer $a_i$ is uniquely determined by
$\cF$ and called the \emph{local semirank} of $\cF$ at $P_i$ since it agrees with $r$ iff $\cF_{P_i}$ is a free $\oo_{P_i}$-module. The
\emph{semirank} of $\cF$ is the tuple $a:=(a_1,\ldots,a_n)$;
\item[(2)] Let $\overline{P}_i$ be the point of $\overline{C}$ which lies over $P_i$, for $1\leq i\leq n$;  and let $\omega_{\cb}$ be the dualizing sheaf on $\cb$. Then here is an equivalence between:
\begin{itemize}
\item[$\bullet$] the category $\mathbb{F}(r,d,a)$ of isomorphism classes of torsion free sheaves on $C$, of rank $r$, degree $d$ and semirank $a$;
\item[$\bullet$] the category $\mathbb{T}(r,d,a)$ of triples $(E,\{V_i\}_{i=1}^n,\{\sigma_i\}_{i=1}^n)$ where
\begin{itemize}
\item[(i)] $E$ is a bundle of rank $r$ and degree $d-nr-\sum_{i=1}^n a_i$ on $\overline{C}$;
\item[(ii)] $V_i$ is a vector subspace of dimension $a_i$ in the direct sum of fibers $E_{(\pb_i)}\oplus (E\otimes\omega_{\cb})_{(\pb_i)}$;
\item[(iii)] $\sigma_i$ is a linear endomorphism in $GL(V_i)$.
\end{itemize}
and where a morphism of triples
$$
\ \ \ \ \ \ \ \ \ \ \ \ \ \ \ \ \ \widetilde{\Phi} : (E,\{V_i\}_{i=1}^n,\{\sigma_i\}_{i=1}^n) \ra (E',\{V_i'\}_{i=1}^n,\{\sigma_i'\}_{i=1}^n)
$$
is a morphism $\Phi: E \ra E'$ satisfying $(\Phi_{(\pb_i)}\oplus(\Phi\otimes 1)_{(\pb_i)})(V_i) \subset V_i'$ such that
$$
\ \ \ \ \ \ \ \ \ \ \ \ \
\xymatrix{V_i \ar[d]_{\Phi_{(\pb_i)}\oplus(\Phi\otimes 1)_{(\pb_i)}} \ar[r]^{\sigma_i} &  V_i \ar[d]^{\Phi_{(\pb_i)}\oplus(\Phi\otimes 1)_{(\pb_i)}}\\
         V_i' \ar[r]^{\sigma_i'} & V_i'}
$$
is a commutative diagram for $i = 1, \dots, n$.
\end{itemize}
\end{itemize}}

It is important to point out that one is able to mix the above
result together with \cite[Thm. 3.2]{ALR} to get a categorical
characterization of torsion free sheaves on curves having double
points -- nodes or cusps -- at most as singularities by means of
triples.

\

\paragraph{\bf Acknowledgments.}
We thank Steven L. Kleiman for some fruitful discussions by email.
The first author is partially supported by CNPq grant number
301618/2008-9 and FAPEMIG grant PPM-00191-09. The third named author
is partially supported by the CNPq grant number 304919/2009-8.

\section{Proof of the Theorem}

We start by proving part (1) of the Theorem by means of the following result.

\begin{lemm}

\label{lemcar}

Let $R$ be a local ring in an algebraic function field of one variable over an algebraically closed ground field $k$. Let also $\overline{R}$ and $\mathfrak{m}$ be, respectively, its integral closure and maximal ideal. If ${\rm dim}_k(\overline{R}/R)=1$ then for any finitely generated torsion free $R$-module $M$ there are uniquely determined integers $a$ and $b$ such that
$$
M\cong R^{\oplus a}\oplus\mathfrak{m}^{\oplus b}.
$$
\end{lemm}

\begin{proof}
Set $K|k$ to be the algebraic function field. Under the hypothesis above we have $k\subset R\subset\overline{R}\subset K$ where $K$ and $k$, respectively, agree with the field of fractions and residue field of $R$. If $M$ splits as above then $a+b=r:={\rm rk}(M\otimes K)$ and $a+2b={\rm rk}(M\otimes k)$ so $a$ and $b$ are uniquely determined by $M$.

Since $\overline{R}$ is a semi-local Dedekind domain, it is a unique factorization domain; so $\overline{M}:=(M\otimes\overline{R})/{\rm Torsion}(M\otimes\overline{R})$ is free. Moreover, one is able to choose a basis for $\overline{M}$ constitued by elements in $M$. Indeed, by Nakayama's Lemma, we need only take elements in $M$ whose residues form a free basis of $\overline{M}/J\overline{M}$ where $J$ is the Jacobson radical of $\overline{R}$. Therefore, we may assume
\begin{equation}
\label{equinb}
R^{\oplus r} \subset M \subset \overline{R}^{\,\oplus r}.
\end{equation}

If so, there is a well defined morphism
$$
\varphi: M \longrightarrow  (\overline{R}/R)^{\oplus r}.
$$
Since ${\rm dim}(\overline{R}/R)=1$ we are able to write $\overline{R}/R=k\bar{t}$ where $\bar{t}$ is the class of a fixed $t \in \overline{R}$. Now we choose a basis for ${\rm im}\varphi$ which thus can be written as $\{ w_1\bar{t},\ldots,w_b\bar{t}\}$ where $w_1,\ldots,w_b$ are linearly independent vectors in $k^r$. Set $a:=r-b$ and choose $v_1,\ldots,v_a$ in $k^r$ such that $B:=\{v_1,\ldots,v_a,w_1,\ldots,w_b\}$ is a basis for $k^r$.

Now $B$ yields a new decomposition of $\overline{R}^{\,\oplus r}$. In fact, for any $f \in \overline{R}^{\,\oplus r}$, there are uniquely determined $g_1,\ldots,g_a,h_1,\ldots,h_b \in \overline{R}$ such that
\begin{equation}
\label{equfgh}
f= g_1v_1+\ldots+g_av_a+h_1w_1+\ldots+h_bw_b
\end{equation}
because, by the very Cramer's Rule, the $g_i$ and $h_i$ are the only solutions of a system of which the column vector has entries in $\overline{R}$ and the main matrix has entries in $k$ and nonzero determinant since $B$ is a basis for $k^r$.

We will prove that
$$
M=Rv_1\oplus\ldots\oplus Rv_a\oplus \overline{R}w_1\oplus\ldots\oplus\overline{R}w_b.
$$

In order to do so, write again any $f \in \overline{R}^{\,\oplus r}$ as in (\ref{equfgh}). We first claim that if $f\in M$, then $g_1,\ldots,g_a \in R$. Indeed, if $f\in M$ then $f$ is of the form
\begin{equation}
\label{equfft}
f=(f_1,\ldots,f_r)+(c_1w_1+\ldots+c_bw_b)t
\end{equation}
where $f_i \in R$ and $c_i \in k$. In fact, clearly
$$
f=(f_1+d_1t,...,f_r+d_rt)
$$
for $f_i \in R$ and $d_i \in k$ because $\overline{R}/R=k\bar{t}$. Now
\begin{align*}
\varphi(f)&=\varphi(f_1+d_1t,\ldots,f_r+d_rt) \\
          &=\varphi(d_1t,\ldots,d_rt) \\
          &=(d_1\bar{t},\ldots,d_r\bar{t}) \\
          &=c_1w_1\bar{t}+\ldots+c_bw_b\bar{t}
\end{align*}
because ${\rm im}\varphi$ is generated by the $w_i\bar{t}$. So
$$
(d_1,\ldots,d_r)=c_1w_1+\ldots+c_bw_b
$$
and $f$ is in fact written as in (\ref{equfft}). The problem now simplifies to the following one: if $f\in R^{\oplus r}$ then the
$g_i$ are in $R$. But this is immediate: the same basis $B$ which yields a new decomposition for $\overline{R}^{\,\oplus r}$, yields a new
decomposition for $R^{\oplus r}$ as well by means of the same argument.

Conversely, given $f\in\overline{R}^{\,\oplus r}$ such that $g_1,\ldots,g_a \in R$ then $f\in M$. Indeed, if so, write $h_i=c_it+f_i$ where $c_i \in k$ and $f_i\in R$, take $m \in M$ such that $\varphi(m)=c_1w_1\bar{t}+\ldots+c_bw_b\bar{t}$, then
$$
m=(f_1',\ldots,f_r')+(c_1w_1+\ldots+c_bw_b)t
$$
with $f_i'\in R$. Set
$$
m'= g_1v_1+\ldots+g_av_a+f_1w_1+\ldots+f_bw_b-(f_1',\ldots,f_r').
$$
Then $m'\in M$ because $A^r \subset M$. Now $f=m+m'$ so $f \in M$.

The result follows from the isomorphism $\overline{R}\cong\mathfrak{m}$ when ${\rm dim}(\overline{R}/R)=1$.
\end{proof}

\begin{rema}
\label{remscd} \emph{The above lemma generalizes \cite[Prp. 2, p.
164]{S}. We could have said that, up to slight changes, the same
proof applies but we did not, due to some relevant reasons. The
proof in \cite{S} is focused on how to get (\ref{equinb}). Here, the
same argument is put within a general form, holding for a local ring
of any singular point in any curve. On the other hand, in \cite{S}
it is stated that the desired decomposition is immediate from
(\ref{equinb}) once we assume the hypothesis and construct a
suitable basis. Actually, the reader should note the ``unusual" way
a basis which works for our purposes is constructed here, and we
still wonder how to get the lemma with the identification
$\overline{R}/R=k$, as done in \cite{S}. At any rate, we hope the
above proof helps the reader understand Seshadri's ideas.}
\end{rema}

Now we prove part (2) of the Theorem. For the sake of simplicity, we
assume the curve has just one singular point. The general case is
straight forward, once this is proved.

\begin{theo}
If $C$ is a curve with only one singular point $P$, which is a cusp,
then the categories $\mathbb{F}(r,d,a)$ and $\mathbb{T}(r,d,a)$ are
equivalent.
\end{theo}

\begin{proof}
Let $\cF$ be a torsion free sheaf on $C$ of semirank $a$ at $P$. Consider the surjective homomorphism
$$ \mathcal{F}_P\longrightarrow k_P^{\oplus a}$$
where $k_P$ denotes the residual field at the point $P$. It yields an exact sequence
\begin{equation}
\label{equext}
0\longrightarrow {\cE}\longrightarrow {\cF}\longrightarrow k_P^{\oplus a} \longrightarrow 0
\end{equation}
where $k_P^{\oplus a}$ is the skyscraper sheaf concentrated at $P$. The kernel $\cE$ is uniquely determined by $\cF$ and satisfies
$$
{\cE}_{P} \simeq \mmp^{\oplus r}.
$$
So $\cF$ can be seen as an extension given by (\ref{equext}). In order to characterize such an extension, let $\cM$ be the sheaf on $C$ defined by
$$ \cM_Q=
\begin{cases}
\mathfrak{m}_P &\ \ {\rm if}\ Q=P \\
\mathcal{O}_Q   &\ \ {\rm if}\ Q\neq P.
\end{cases}
$$
We claim that
\begin{equation}
\label{equexc}
{\Ext}^1(k_P^{\oplus a},{\cE})\simeq {\Hom}(\cM^{\oplus a},\cE)/{\rm Hom}(\oo_C^{\oplus a},{\cE}).
\end{equation}
In fact, we may assume $h^1(\mathcal{E})=0$ twisting $\mathcal{E}$ if needed. The exact sequence
$$
0\longrightarrow\cM^{\oplus a} \stackrel{i}\longrightarrow \oo_C^{\oplus a} \longrightarrow k_P^{\oplus a}\longrightarrow 0
$$
yields the long exact sequence
$$
{\rm Hom}(k_P^{\oplus a},{\cE}) \longrightarrow {\Hom}(\oo_C^{\oplus a},\cE) \longrightarrow {\Hom}(\cM^{\oplus a},\cE) \longrightarrow
$$ $$
\longrightarrow
{\Ext}^1(k_P^{\oplus a},{\cE}) \longrightarrow {\Ext}^1(\oo_C^{\oplus a},{\cE}).
$$
But ${\rm Hom}(k_P^{\oplus a},{\cE})=0$ since $k_P^{\oplus a}$ is a torsion sheaf, and ${\Ext}^1(\oo_C^{\oplus a},{\cE})\simeq H^1(\cE)^{\oplus a}=0$
since $h^1(\mathcal{E})=0$. So (\ref{equexc}) is proved. It corresponds to the bottom level of the diagram
\begin{tiny}
\begin{equation}\label{equdia}
\xymatrix@C-0.5pc@R-0.5pc{
\oo_{\cb}(-2\pb)^{\oplus a} \ar@{-}[dd]\ar[dr]^{\overline{\phi}}\ar[rr]&&
\oo_{\cb}^{\oplus a} \ar@{-}[dd]\ar[dr]^{\overline{\phi}\otimes 1}\ar[dl]^{\overline{\phi}'}\ar[drrr]^{\varphi} &&
\\
&
E \ar@{-}[dd]\ar[rr] &&
E(2\pb) \ar@{.}[dd]\ar[rr] &&
(\oo_{\pb}/\mathfrak{m}_{\pb}^2)\otimes E_{\pb} \ar@{.}[dd]
\\
\cM^{\oplus a}\ar[rr]^{i}\ar[dr]^{\phi} &&
\oo_{C}^{\oplus a} \ar[dr]\ar[rr]\ar[dl]^{\phi'} &&
k_{P}^{\oplus a} \ar@{=}[dr]
\\
&
\cE \ar[rr]\ar@{=}[dr] &&
\cF~ \ar[dr] \ar[rr] &&
k_{P}^{\oplus a} \ar[dr]^{\psi}\ar[dl]^{\psi'}
\\
&& \cE \ar[rr] && I_0 \ar[rr]^{d} && I_1' }
\end{equation}
\end{tiny}

\noindent where $0\to\cE\to I_0 \stackrel{d}\to I_1\to\cdots$ is an injective resolution and $I_1':={\rm im}\,d$. So (an isomorphism class of) a torsion
 free sheaf $\cF$ on $C$ is determined by a morphism $\phi: \cM^{\oplus a}\to\cE$ (up to a $\phi': \oo_C^{\oplus a}\to\cE$) or,
 equivalently, by a morphism $\psi:k_{P}^{\oplus a}\to I_1'$ (up to a $\psi':k_{P}^{\oplus a}\to I_0$); the sheaf $\cF$ is the pushout
 of $i$ and $\phi$ or, equivalently, the pullback of $d$ and $\psi$.

Now we build the top part of diagram (\ref{equdia}), which will
provide the functor we are looking for. Let $\pi:\cb\to C$ be the
normalization map and let $E$ be the bundle on $\cb$ such that
$$
\cE=\pi_*(E).
$$
We have that ${\Hom}(\oo_C^{\oplus a},\cE)={\Hom}(\oo_C^{\oplus a},\pi_{*}(E))\simeq {\Hom}(\oo_{\cb}^{\oplus a},E)$ and we also have
\begin{align*}
{\Hom}(\cM^{\oplus a},\cE) &={\Hom}(\cM^{\oplus a},\pi_{*}(E))\simeq {\Hom}(\pi^{*}(\cM^{\oplus a}),E) \\
                &\simeq{\Hom}(\pi^{*}(\cM^{\oplus a})/{\rm Torsion}(\pi^{*}(\cM^{\oplus a})),E) \\
                & \simeq {\Hom}(\oo_{\cb}(-2\pb)^{\oplus a},E)
\end{align*}
by adjunction formula and by the fact that ${\Hom}({\rm Torsion}(\pi^{*}(\cM^{\oplus a})),E)=0$. Hence
\begin{equation}
\label{equgrx}
{\Hom}(\cM^{\oplus a},\cE)/{\Hom}(\oo_C^{\oplus a},\cE)\simeq {\Hom}(\oo_{\cb}(-2\pb)^{\oplus a},E)/{\Hom}(\oo_{\cb}^{\oplus a},E).
\end{equation}
So $\cF$ is also determined by $\overline{\phi}: \mathcal{O}_{\cb}(-2\pb)^{\oplus a}\to E$ (up to a $\overline{\phi}':\mathcal{O}_{\cb}^{\oplus a}\to E$).
But we may write
\begin{equation}
\label{equgrt}
{\Hom}(\oo_{\cb}(-2\pb)^{\oplus a},E)/{\Hom}(\oo_{\cb}^{\oplus a},E)\simeq {\Hom}(\oo_{\cb}^{\oplus a},E(2\pb))/{\Hom}(\oo_{\cb}^{\oplus a},E)
\end{equation}
switching $\overline{\phi}$ to $\overline{\phi}\otimes 1$ and same
target to same source.


Now, from
$$
0\longrightarrow \oo_{\cb}(-2\pb)\longrightarrow \oo_{\cb}\longrightarrow  \oo_{\pb}/\mathfrak{m}_{\pb}^2 \longrightarrow 0
$$
tensored by $E(2\pb)$ we get
$$
0\longrightarrow E\longrightarrow E(2\pb)\longrightarrow (\oo_{\pb}/\mathfrak{m}_{\pb}^2)\otimes E_{\pb}\longrightarrow 0
$$
where $E_{\pb}$ is the stalk at $\pb$. The above sequence yields the long exact sequence
$$
0 \longrightarrow {\rm Hom}(\oo_{\cb}^{\oplus a},E) \longrightarrow {\Hom}(\oo_{\cb}^{\oplus a},E(2\pb))\longrightarrow
$$
$$
\longrightarrow {\Hom}(\oo_{\cb}^{\oplus a},(\oo_{\pb}/\mathfrak{m}_{\pb}^2)\otimes E_{\pb}) \longrightarrow {\Ext}^1(\oo_{\cb}^{\oplus a},E).
$$
Now ${\Ext}^1(\oo_{\cb}^{\oplus a},E)=H^1(E)=H^1(\pi_*(E))=H^1(\cE)$
which we may assume to vanish and we get
\begin{equation}
\label{equlon} {\Hom}(\oo_{\cb}^{\oplus a},E(2\pb))/{\rm
Hom}(\oo_{\cb}^{\oplus a},E) \simeq {\Hom}(\oo_{\cb}^{\oplus
a},(\oo_{\pb}/\mathfrak{m}_{\pb}^2)\otimes E_{\pb});
\end{equation}
and hence $\cF$ is determined by a morphism
$\varphi:\oo_{\cb}^{\oplus
a}\to(\oo_{\pb}/\mathfrak{m}_{\pb}^2)\otimes E_{\pb}$ and diagram
(\ref{equdia}) pictured above is now complete. It establishes the
equivalences
\begin{equation}
\label{equeqv}
[\cF]\ \leftrightarrow\ \psi\ ({\rm mod}\ \psi')\ \leftrightarrow\ \phi\ ({\rm mod}\ \phi')\ \leftrightarrow\ \overline{\phi}\ ({\rm mod}\ \overline{\phi}')\ \leftrightarrow\ \overline{\phi}\otimes 1\ ({\rm mod}\ \overline{\phi}')\ \leftrightarrow\ \varphi
\end{equation}
given by (\ref{equext}), (\ref{equexc}), (\ref{equgrx}),
(\ref{equgrt}) and (\ref{equlon}). Note that in (\ref{equdia}), the
vertical full lines correspond to pull back operations modded out by
torsion; however, the quotient $\pi^{*}(\cF)/{\rm
Torsion}(\pi^{*}(\cF))$ is just a subsheaf of $E(2\pb)$ unless $\cF$
is a bundle. That's why the correspondent vertical line is dotted.

Now let $k_{\pb}$ be the residual field at $\pb$. Use the natural
isomorphism
$$
{\Hom}(\oo_{\cb}^{\oplus a},(\oo_{\pb}/\mathfrak{m}_{\pb}^2)\otimes E_{\pb})\simeq{\Hom}(k_{\pb}^{\oplus a},(\oo_{\pb}/\mathfrak{m}_{\pb}^2)\otimes E_{\pb})
$$
in order to consider $\varphi:k_{\pb}^{\oplus
a}\to(\oo_{\pb}/\mathfrak{m}_{\pb}^2)\otimes E_{\pb}$. We claim that
if a coherent sheaf $\cF$ fits into a diagram like (\ref{equdia}),
then it is torsion free if and only if $\varphi$ is injective.  To
prove this, first, let $\{e_{i}\}_{i=1}^{a}$ denote the standard
basis of $k^{\oplus a}$, take $E_{\pb}=\oplus_{j=1}^r \oo_{\pb}s_i$
and write
$$
\varphi(e_i)=\bigoplus_{j=1}^r\ (a_{ij}\bar{1}+b_{ij}\bar{t})\otimes s_j
$$
with $a_{ij},b_{ij}\in k_{\pb}\simeq k$. Now $\cF$ satisfies the following local commutative diagram
$$
\xymatrix{
0\ar[r] & \mmp^{\oplus a}\ar[r]^{i} \ar[d]^{\phi_P} & \op^{\oplus a} \ar[r]\ar[d] &
k_P\ar@{=}[d]  \ar[r] & 0 \\
0\ar[r] & \mmp^{\oplus r}\ar[r] & \cF_P \ar[r] & k_P \ar[r] & 0
}$$
and $\cF_P$ is the pushout of $\phi_P$ and $i$. Hence
$$
\cF_P\simeq \oo_P^{\oplus a}\oplus\mmp^{\oplus r}/\{(x,-\phi_P(x))\,|\, x\in\mmp^{\oplus a}\}.
$$
But note that $\phi_P$ is, by construction, closely related to the $\varphi$ above. Indeed, one may write
$$
\phi_P(t^2e_{i})=\bigoplus_{j=1}^r\ a_{ij}t^2+b_{ij}t^3+h_{ij}t^4
$$
where $t$ is a local parameter for the integral closure $\obp\simeq\oo_{\pb}$ and $h_{ij}\in\obp$. Assume $\varphi$ is not injective. Then for a nonzero $v=(v_1,\dots,v_a)\in k_{\pb}^{\oplus a}\simeq k^{\oplus a}\subset\op^{\oplus a}$ we have $\varphi(v)=0$. Thus clearly $\overline{(v,-t^2(\sum_{i=1}^a\oplus_{j=1}^r v_ih_{ij}))}\neq 0$ in $\cF_P$ because $v\not\in\mmp^{\oplus a}$. But
$$
t^2\overline{\left(v,-t^2\left(\sum_{i=1}^a\bigoplus_{j=1}^r v_ih_{ij}\right)\right)} =\overline{\left(t^2v,-t^4\left(\sum_{i=1}^a\bigoplus_{j=1}^r v_ih_{ij}\right)\right)} =\overline{(t^2v,-\phi_P(t^2v))}=0
$$
hence $\cF_P$ has torsion and so does $\cF$. , suppose
$z\overline{(x,y)}=0$. Then $zy=-\phi_P(zx)$. If $z\neq 0$ and
$x\in\mmp^{\oplus a}$ then $y=-\phi_P(x)$ which implies
$\overline{(x,y)}=0$. Thus $\cF_P$ admits torsion only if
$$ c\,t^m+d\,t^{m+1}+ht^{m+2}\overline{\left(\left(\bigoplus_{i=1}^a
v_i+h_it^2\right),\left(\bigoplus_{j=1}^r
w_jt^{n_j}+g_{j}t^{n_j+1}\right)\right)}=0
$$
for some $v_i,w_j,c,d\in k$ and $h,h_i,g_j\in\obp$ where $c\neq 0$,
every $w_j\neq 0$, all $n_j\geq 2$ and $v_{i'}\neq 0$ for some
$i'\in\{1,\dots,a\}$. Writing $\phi_P=(\phi_{P1},\dots,\phi_{Pr})$,
then for every $j\in\{1,\dots,r\}$ we have
\begin{align*}
cw_jt^{m+n_j}+\dots &=-\phi_{Pj}\left(\bigoplus_{i=1}^a cv_i\,t^m+dv_i\,t^{m+1}+\dots\right) \\
                    &=-c\left(\sum_{i=1}^a v_ia_{ij}\right)t^m-\left(c\sum_{i=1}^a v_ib_{ij}+d\sum_{i=1}^a v_ia_{ij}\right)t^{m+1}+\dots
\end{align*}
and since $n_j\geq 2$ we get
$$
\sum_{i=1}^a v_ia_{ij}=\sum_{i=1}^a v_ib_{ij}=0
$$
which means that for $v=(v_1,\dots,v_a)$ we have $\varphi(v)=0$. But
$v\neq 0$ since $v_{i'}\neq 0$. So $\varphi$ is not injective.

Now consider
\begin{equation}
\label{equwww}
W:=(\oo_{\pb}/\mathfrak{m}_{\pb}^2)\otimes_{\oo_{\pb}}
E_{\pb}=(\oo_{\pb}/\mmpb\oplus_{k_{\pb}}\mmpb/\mmpb^2)\otimes_{\oo_{\pb}}E_{\pb}.
\end{equation}
Note that we are not able to use the distributive law here. On the
other hand, $W$ is a vector space over $k_{\pb}$ of dimension $2r$
with basis $\{\bar{1}\otimes s_i ,\bar{t}\otimes s_i\}_{i=1}^r$.
Moreover, $\oo_{\pb}/\mmpb\otimes_{\oo_{\pb}}E_{\pb}$ is by
definition the fiber $E_{(\pb)}$. Besides $\mmpb/\mmpb^2$ is
naturally associated to the fiber of the dualizing sheaf
$\omega_{\cb}$ at $\pb$. This leads us to consider the following
linear isomorphism of vector spaces over $k_{\pb}$, defined on
generators as
\begin{gather}
\label{eqummm}
\begin{matrix}
\mu : & W & \longrightarrow & E_{(\pb)}\oplus_{k_{\pb}} (E\otimes_{\oo_{\cb}}\omega_{\cb})_{(\pb)}\\
         & \bar{1}\otimes s_i            & \longmapsto     & (s_i\otimes 1)\oplus 0 \\
         & \bar{t}\otimes s_i            & \longmapsto     & 0\oplus (s_i\otimes dt\otimes 1).
\end{matrix}
\end{gather}

This enables us to finally take care of the other two items in the
triple we are aiming to build, besides the bundle $E$ we have
already defined. In fact, for a given torsion free sheaf $\cF$ of
semirank $a$ on $C$, consider the injective
$\varphi\in{\Hom}(k_{\pb}^{\oplus a},W)$ associated to it. Set
$$
V:=\mu\circ\varphi(k_{\pb}^{\oplus a})\subset E_{(\pb)}\oplus
(E\otimes\omega_{\cb})_{(\pb)}.
$$
Now given any $a$-dimensional subspace $V$ of $E_{(\pb)}\oplus
(E\otimes\omega_{\cb})_{(\pb)}$ one may fix a priori a basis
$\{v_i\}_{i=1}^a$ and consider the isomorphism
\begin{gather*}
\begin{matrix}
\nu : & V & \longrightarrow & k_{\pb}^{\oplus a} \\
         & v_i           & \longmapsto     & e_i
\end{matrix}
\end{gather*}
This allows us to define the automorphism
$$
\sigma:=\mu\circ\varphi\circ\nu:V\longrightarrow V
$$
which is in $GL(V)$ since $\varphi$ is injective and $\mu,\nu$ are
bijective.

We are ready to define the functor we are looking for. Set
\begin{gather*}
\begin{matrix}
\mathfrak{F} : &{\rm Ob}(\mathbb{F}(r,d,a)) & \longrightarrow & {\rm Ob}(\mathbb{T}(r,d,a))\\
               &  [\cF]             & \longmapsto     & (E,V,\sigma)
\end{matrix}
\end{gather*}
It is well defined since, for $g$ and $\overline{g}$ denoting the
arithmetic genus of $C$ and $\overline{C}$, we have
\begin{align*}
\deg_{\cb}E &=\deg_{C}\cE -r(g-\overline{g}) \\
            &=(\deg_{C}\cF+a)-r \\
            &=d+a-r
\end{align*}

Now we define how the functor acts on morphisms. First, recall from the statement of the theorem in the introduction, that a morphism of triples
$$
\ \ \ \ \ \ \ \ \ \ \ \ \ \ \ \ \ \widetilde{\Phi} : (E,V,\sigma) \ra (E',V',\sigma')
$$
is a morphism $\Phi: E \ra E'$ satisfying $(\Phi_{(\pb)}\oplus(\Phi\otimes 1)_{(\pb)})(V) \subset V'$ such that
\begin{equation}
\label{equcom}
\ \ \ \ \ \ \ \ \ \ \ \ \
\xymatrix{V \ar[d]_{\Phi_{(\pb)}\oplus(\Phi\otimes 1)_{(\pb)}} \ar[r]^{\sigma} &  V \ar[d]^{\Phi_{(\pb)}\oplus(\Phi\otimes 1)_{(\pb)}}\\
         V' \ar[r]^{\sigma'} & V'}
\end{equation}
is a commutative diagram.

So let $f: \cF \ra \cF'$ be a homomorphism of torsion free sheaves
on $C$ and set $\mathfrak{F}(\cF) := (E,V,\sigma)$ and $\mathfrak{F}(\cF') := (E',V',\sigma')$. If
$\cE \subset \cF$ and $\cE' \subset \cF'$ are the subsheaves as in (\ref{equext}), then clearly $f$ maps $\cE$ into $\cE'$. It lifts to a morphism of vector bundles
$$
\Phi:=\pi^*(f|\cE)/\text{torsion}: E\longrightarrow E'.
$$
on the normalization.

We have the following commutative diagram
\begin{eqnarray} \label{diag2}
\xymatrix{      0 \ar[r] & \cE \ar[r] \ar[d]^{f|\cE} &  \cF \ar[d]^{f} \ar[r] & k_P^{\oplus a} \ar[r] \ar[d]^{f_{P}} & 0\\
                0 \ar[r] & \cE' \ar[r] & \cF' \ar[r] & k_P^{\oplus a'} \ar[r] & 0.         }
\end{eqnarray}
and an induced map $f_{\pb}:k_{\pb}^{\oplus a }\to k_{\pb}^{\oplus a'}$. Set $W':=(\oo_{\pb}/\mathfrak{m}_{\pb}^2)\otimes
E'_{\pb}$ as in (\ref{equwww}) and $\mu' : W' \to E'_{(\pb)}\oplus (E'\otimes\omega_{\cb})_{(\pb)}$ as in (\ref{eqummm}).
Let $\varphi\in\Hom(k_{\pb}^{\oplus a},W)$ and $\varphi'\in\Hom(k_{\pb}^{\oplus a'},W')$ denote the
 homomorphisms associated to the extensions of diagram \eqref{diag2} according to (\ref{equeqv}). By
definition the following diagram commutes
\begin{eqnarray*}
\xymatrix{
 k_{\pb}^{\oplus a} \ar[d]_{f_{\pb}} \ar[r]^{\varphi}  &    W \ar[r]^{\mu\ \ \ \ \ \ \ \ \ \ \ \ } & E_{(\pb)}\oplus (E\otimes\omega_{\cb})_{(\pb)}  \ar[d]^{\Phi_{(\pb)}\oplus(\Phi\otimes 1)_{(\pb)}}\\
 k_{\pb}^{\oplus a'}\ar[r]^{\varphi'}   &  W' \ar[r]^{\mu'\ \ \ \ \ \ \ \ \ \ \ \  } & E'_{(\pb)}\oplus (E'\otimes\omega_{\cb})_{(\pb)}         }
\end{eqnarray*}
Since $V:=\mu\circ\varphi(k_{\pb}^{\oplus a})$ and $V':=\mu'\circ\varphi(k_{\pb}^{\oplus a'})$ we have $(\Phi_{(\pb)}\oplus(\Phi\otimes 1)_{(\pb)})(V) \subset V'$. The fact that the diagram (\ref{equcom})
commutes for any choice of basis $\sigma$ and $\sigma'$ in $V$ and $V'$ is straightforward.

The functor $\mathfrak{F}$ is now defined both on objects and morphisms. By construction it is bijective on objects, so it remains to prove that $\mathfrak{F}$ is fully faithful.

So let $\widetilde{\Phi}: (E,V,\sigma) \ra (E',V', \sigma')$ be a homomorphism
of triples on $\cb$. Let $\cF$ and $\cF'$ be torsion free
sheaves on $C$ with $\Psi([\cF]) = (E,V, \sigma)$
and $\Psi([\cF']) = (E',V', \sigma')$.

Note that both $\mu\circ\varphi$ and $\mu'\circ\varphi'$ are linear isomorphisms onto $V$ and $V'$,
so $\Phi$ induces a map $f_{\pb}:k_{\pb}^{\oplus a}\to k_{\pb}^{\oplus a'}$ and hence a
map $f_{P}:k_{P}^{\oplus a}\to k_{P}^{\oplus a'}$.  This yields a map $\tilde{f}_P:\mathcal{M}^{\oplus a}\to\mathcal{M}^{\oplus a'}$,
as in the diagram
\begin{tiny}
$$\xymatrix{
\mathcal{M}^{\oplus a} \ar[dd] \ar@{^{(}->}[rr] \ar[dr]^{\tilde{f}_P}
&&
\mathcal{O}_C^{\oplus a} \ar[dd] \ar@{->>}[rr]
&&
k_P^{\oplus a}\ar@{=}[dd] \ar[drr]^{f_P}
\\
&
\mathcal{M}^{\oplus a'} \ar[dd] \ar@{^{(}->}[rr]
&&
\mathcal{O}_C^{\oplus a'}\ar[dd] \ar@{->>}[rrr]
&&&
k_P^{\oplus a'} \ar@{=}[dd]
\\
\cE \ar@{^{(}->}[rr] \ar[dr]^{\pi_*(\Phi)}
&&
\cF \ar@{->>}[rr] \ar@{-->}[dr]^{f}
&&
k_P^{\oplus a}\ar[drr]^{f_P}
\\
&
\cE' \ar@{^{(}->}[rr] && \cF' \ar@{->>}[rrr] &&& k^{\oplus a'} }$$
\end{tiny}
The universal property
of the push-out gives us a homomorphism
$$
f: \cF \ra \cF'
$$
such that the whole diagram commutes. It is uniquely determined as a map of extensions so $\mathfrak{F}$ is fully faithful.
\end{proof}

\end{document}